\documentclass[a4paper,12pt]{amsart}

\usepackage{fullpage}
\usepackage[latin2]{inputenc}
\usepackage[english]{babel}
\usepackage{tikz}
\usepackage{amsmath}
\usepackage{amsthm}
\usepackage{amssymb}
\usepackage{cite}
\usepackage{amsrefs}
\usepackage{rotating}
\usepackage{graphicx}
\usepackage{subfigure}
\usepackage{bbold}

\newtheorem{thm}{Theorem}[section]
\newtheorem{conj}[thm]{Conjecture}
\newtheorem{lem}[thm]{Lemma}
\newtheorem{defi}[thm]{Definition}
\newtheorem{corr}[thm]{Corollary}

\theoremstyle{remark}
\newtheorem*{remark}{Remark}

\theoremstyle{definition}
\newtheorem*{acknowledgement}{Acknowledgements}

\newcommand{\eps}{\varepsilon}
\newcommand{\csillag}{\sideset{}{^\ast}}
\newcommand{\F}{\mathbb{F}_p}

\begin{document}

\title{New bounds on even cycle creating Hamiltonian paths using expander graphs}

\makeatother
\author{Gergely Harcos}
\author{Daniel Soltész}
\address{Alfréd Rényi Institute of Mathematics, Hungarian Academy of Sciences, POB 127, Budapest H-1364, Hungary}
\email[Gergely Harcos]{gharcos@renyi.hu}
\email[Daniel Soltész]{solteszd@renyi.hu}
\address{MTA Rényi Intézet Lendület Automorphic Research Group}
\email[Gergely Harcos]{gharcos@renyi.hu}
\address{Central European University, Nador u. 9, Budapest H-1051, Hungary}
\email[Gergely Harcos]{harcosg@ceu.edu}
\thanks{First author supported by NKFIH (National Research, Development and Innovation Office) grant K~119528, and by the MTA Rényi Intézet Lendület Automorphic Research Group. Second author supported by NKFIH (National Research, Development and Innovation Office) grants K~108947, K~120706, KH~126853, KH~130371.}

\begin{abstract}
We say that two graphs on the same vertex set are $G$-creating if their union (the union of their edges) contains $G$ as a subgraph. Let $H_n(G)$ be the maximum number of pairwise $G$-creating Hamiltonian paths of $K_n$. Cohen, Fachini and Körner proved
\[n^{\frac{1}{2}n-o(n)}\leq H_n(C_4) \leq n^{\frac{3}{4}n+o(n)}.\]
In this paper we close the superexponential gap between their lower and upper bounds by proving
\[n^{\frac{1}{2}n-\frac{1}{2}\frac{n}{\log{n}}-O(1)}\leq H_n(C_4) \leq n^{\frac{1}{2}n+o\left(\frac{n}{\log{n}} \right)}.\]
We also improve the previously established upper bounds on $H_n(C_{2k})$ for $k>3$, and we present a small improvement on the lower bound of Füredi, Kantor, Monti and Sinaimeri on the maximum number of so-called pairwise reversing permutations. One of our main tools is a theorem of Krivelevich, which roughly states that (certain kinds of) good expanders contain many Hamiltonian paths.
\end{abstract}

\maketitle

\section{Introduction}

There are many results concerning the size of the largest set of permutations that satisfy some prescribed binary relation, see \cites{intsurvey,additive,tintersect,frankldeza,reversefree,colliding,KMS}. There is a natural ($2$-to-$1$) correspondence between permutations of $[n]$ and (undirected) Hamiltonian paths of the complete graph $K_n$. The main questions studied in this paper are among the first natural questions that arise when one considers irreflexive relations between Hamiltonian paths of $K_n$.

\begin{defi}
We say that two graphs on the same vertex set are $G$-creating if their union (the union of their edges) contains $G$ as a not necessarily induced subgraph. Let $H_n(G)$ and $\overline{H_n(G)}$ be the maximum number of pairwise $G$-creating and pairwise non-$G$-creating Hamiltonian paths of $K_n$, respectively.
\end{defi}

The study of $H_n(G)$ for various graphs $G$ was initiated in \cite{komesi}. There it was observed that the maximum number of Hamiltonian paths of $K_n$ such that each pairwise union contains an odd cycle (of unspecified length) is $\binom{n}{\lfloor n/2 \rfloor}$ if $n$ is odd and $\binom{n}{\lfloor n/2 \rfloor}/2$ if $n$ is even. The authors of \cite{komesi} asked whether replacing an odd cycle of unspecified length with a triangle would result in the same answer. This question was answered affirmatively by I. Kovács and the second author in \cite{triangle}.

\begin{thm}[Kovács--Soltész~\cite{triangle}]\label{thm:triangle}
We have
\[H_n(C_3)= \begin{cases} \binom{n}{\left \lfloor \frac{n}{2} \right \rfloor} & \text{when $n$ is odd;} \\
\frac{1}{2}\binom{n}{\left \lfloor \frac{n}{2} \right \rfloor} & \text{when $n$ is even.}\end{cases}\]
\end{thm}

The ideas of Theorem~\ref{thm:triangle} were generalized in \cite{oddcycle} where non-trivial lower bounds were obtained for $H_n(C_{2k+1})$. In particular, when $k$ is a power of two, then $H_n(C_{2k+1})=2^{n+o(n)}$. Although the ideas that led to Theorem~\ref{thm:triangle} have been useful in the case of longer odd cycles, they did not contribute to our understanding of $H_n(C_k)$ for even $k$. The authors of \cite{komesi} proved that the maximum number of Hamiltonian paths of $K_n$ each of whose pairwise unions contains an even cycle (of unspecified length) is $\Omega\left(\frac{n!}{n^2} \right)$. The natural question whether $H_n(C_4)$ behaves in the same way was answered negatively by Cohen, Fachini and Körner.

\begin{thm}[Cohen--Fachini--Körner~\cite{k=4}]\label{thm:oldk=4}
We have
\[\left\lfloor \frac{n}{2} \right\rfloor ! \leq H_n(C_4) \leq n^{\frac{3}{4}n+O\left(\frac{n}{\log{n}}\right)}.\]
\end{thm}

Although Theorem~\ref{thm:oldk=4} clearly shows that the order of magnitude of $H_n(C_4)$ is much smaller than $\frac{n!}{n^2}$, it does not reveal the asymptotic growth rate of $H_n(C_4)$. While the second author tried to improve the bounds of Theorem~\ref{thm:oldk=4}, he managed to apply the method of Cohen, Fachini and Körner to longer even cycles.

\begin{thm}[Soltész~\cite{evencycle}]\label{thm:oldevenk}
For every integer $k>1$, we have
\[n^{\frac{1}{k}n-o(n)} \leq H_n(C_{2k}) \leq n^{\left( 1-\frac{2}{3k^2-2k} \right)n+o(n)}.\]
Moreover, for $k>5$ odd, we have the improved upper bound
\[H_n(C_{2k}) \leq n^{\left( 1-\frac{2}{3k^2-3k} \right)n+o(n)},\]
while for $k\in\{3,5\}$ we have the improved upper bound
\[H_n(C_{2k}) \leq n^{\left( 1-\frac{1}{k^2} \right)n+o(n)}.\]
\end{thm}

In this paper we have finally succeeded to improve the bounds of Theorem~\ref{thm:oldk=4}.

\begin{thm}\label{thm:main}
We have
\[n^{\frac{1}{2}n-\frac{1}{2}\frac{n}{\log{n}}-O(1)}\leq H_n(C_4) \leq n^{\frac{1}{2}n+o\left(\frac{n}{\log{n}} \right)}.\]
\end{thm}

Theorem~\ref{thm:main} closes the superexponential gap of Theorem~\ref{thm:oldk=4}, and determines the leading term in the asymptotics of $\log(H_n(C_k))$ for the smallest unsolved cycle, $C_4$. The improvement on the lower bound is only a modest contribution, the main result is the new upper bound. The main idea in the proof of the upper bound is the observation that the large spectral gap of some dense, regular and $C_4$-free graphs combined with the following theorem of Krivelevich guarantees the existence of $n^{\frac{1}{2}n-o(n)}$ pairwise non-$C_4$-creating Hamiltonian paths. This gives a lower bound on $\overline{H_n(C_4)}$, which can be used readily to obtain an upper bound for $H_n(C_4)$, see Lemma~\ref{lem:vertextransitive}.

\begin{thm}[Krivelevich~\cite{krivelevich}]\label{thm:kriv}
Let $G$ be a $d$-regular graph on $n$ vertices whose non-trivial eigenvalues have absolute value at most $\lambda$. Assume that the following two conditions hold:
\begin{itemize}
\item there is an $\eps>0$ such that $\frac{d}{\lambda} \geq (\log{n})^{1+\eps}$;
\item $ \lim\limits_{n \to \infty} \frac{\log{d}\log{\frac{d}{\lambda}}} {\log{n}}=\infty$.
\end{itemize}
Then the number of Hamiltonian cycles in $G$ is $n!\left(\frac{d}{n}\right)^n(1+o(1))^n$.
\end{thm}

The same ideas applied to $C_{2k}$-free graphs instead of $C_4$-free graphs yield the following improvement on Theorem~\ref{thm:oldevenk}.

\begin{thm}\label{thm:longevencycle}
For every integer $k>2$, we have
\[n^{\frac{1}{k}n-o(n)}\leq H_n(C_{2k})\leq n^{\left(1-\frac{1}{3k} \right)n+o(n)}.\]
\end{thm}

Theorem~\ref{thm:longevencycle} improves on the bounds of Theorem~\ref{thm:oldevenk} except in the case when $k=3$. The authors believe that the proof of Theorem~\ref{thm:main} can be adapted to the case of $C_6$ to give an upper bound of $n^{\frac{2}{3}n+o(n)}$.

The paper is organized as follows. Section~\ref{sec:lower} contains constructions for the lower bound part of Theorem~\ref{thm:main}, while Section~\ref{proof} treats the corresponding upper bound, with some lengthy but straightforward arguments postponed to Appendix~\ref{app:combinatorics}. We prove Theorem~\ref{thm:longevencycle} in Section~\ref{sec:longer}. We conclude with the connection of $H_n(C_4)$-like problems and the so-called reversing permutation conjecture (for which we present an improved lower bound) in Section~\ref{sec:connection}, along with some closing remarks.

\begin{acknowledgement} We are grateful to the referees for their careful reading and valuable comments. We also thank Zoltán Füredi and Miklós Simonovits for useful discussions.
\end{acknowledgement}

\section{The lower bound of Theorem~\ref{thm:main}}\label{sec:lower}

The original construction of Cohen, Fachini and Körner for the lower bound for $H_n(C_4)$ is as follows. For the sake of simplicity, let $n-1$ be even, and let $\pi$ be a permutation of the set $\{2,4,\dotsc,n-1\}$. Observe that the Hamiltonian paths of $K_n$ of the form \[(1,\pi(2),3,\pi(4),5,\dotsc,n-2,\pi(n-1),n),\]
with $\pi$ as above, are pairwise $C_4$-creating. Indeed, if $\pi_1$ and $\pi_2$ are different permutations as above, say $\pi_1(2i) \neq \pi_2(2i)$, then the set of vertices $\{2i-1,\pi_1(2i),2i+1,\pi_2(2i)\}$ induces a $C_4$ in the union of the two paths. The improvement uses the same idea but in a recursive manner.

\begin{lem}\label{claim:newlower}
We have
\[\prod_{i=1}^{\lfloor (n-1)/2 \rfloor}(n-2i)\leq H_n(C_4).\]
\end{lem}

\begin{proof}
We build a set of Hamiltonian paths recursively, by starting from $(*,*,\dotsc,*)$ and placing the yet unused elements of $[n]$ one-by-one somewhere on the path. Initially, let the path be $(1,*,2,*,*,\dotsc,*)$, and let $S_2:=[n]\setminus \{1,2\}$ be the set of remaining elements that we can still place somewhere. Assuming that we have already placed $i$ elements on the path, we perform one of two steps:
\begin{enumerate}
\item If the stars of the currently constructed Hamiltonian path do not form a consecutive block or there is a single star left: we pick any element from $S_i$ and replace the leftmost star with it. Let $S_{i+1}$ be $S_i$ minus the picked element.
\item If the stars of the currently constructed Hamiltonian path do form a consecutive block of size at least 2: we pick the smallest element of $S_i$ and replace the star that is next to the leftmost star with it. Let $S_{i+1}$ be $S_i$ minus the picked element.
\end{enumerate}
It is easy to see that the number of Hamiltonian paths that can be constructed this way is the claimed amount. The fact that two such Hamiltonian paths $H_1$ and $H_2$ are $C_4$-creating can be seen by considering the first step of the procedure at which the two paths become different. Such a step must be of type $(1)$, and hence the picked element is placed between two other already picked elements: $a,b$ which are the same two elements since up to this point $H_1$ was equal to $H_2$. But as in the original construction, these four elements form a $C_4$ in the union $H_1\cup H_2$.
\end{proof}

Using Stirling's formula, it is straightforward to show that Lemma~\ref{claim:newlower} implies the lower bound of Theorem~\ref{thm:main}.

\section{The upper bound of Theorem~\ref{thm:main}}\label{proof}

For many problems where the maximum number of pairwise compatible objects is the question, one can prove the equivalent of the following lemma.

\begin{lem}\label{lem:vertextransitive}
For every graph $G$ and every integer $n>1$, we have
\[H_n(G)\overline{H_n(G)} \leq n!/2.\]
\end{lem}

\begin{proof}
Consider the graph $G'$ whose vertices are the Hamiltonian paths of $K_n$, and two vertices are connected if the corresponding Hamiltonian paths are not $G$-creating. This graph is vertex-transitive since relabeling the vertices of $K_n$ is an automorphism acting transitively on the Hamiltonian paths of $K_n$. Therefore, a well-known result \cite[Lemma~7.2.2]{godsil} shows that
\[\alpha(G')\omega(G')\leq\alpha(G')\omega^\ast(G')=|V(G')|=n!/2,\]
where $\omega^\ast(G')$ denotes the fractional clique number of $G'$. The product on the left-hand side equals $H_n(G)\overline{H_n(G)}$, so we obtained the bound in the lemma.
\end{proof}

We prove the upper bound of Theorem~\ref{thm:main} by establishing a lower bound for $\overline{H_n(C_4)}$ and combining it with Lemma~\ref{lem:vertextransitive}. We construct a large set of pairwise non-$C_4$-creating Hamiltonian paths by proving that a certain $C_4$-free graph contains enough of them. Thankfully, all the heavy lifting is already done. Namely, by Theorem~\ref{thm:kriv}, it is enough to find a dense, regular $C_4$-free graph with a large spectral gap, and it is well-known that some of the densest $C_4$-free graphs are regular with a large enough spectral gap (for the purpose of Theorem~\ref{thm:kriv}). We only have to deal with the fact that these dense $C_4$-free graphs are not simple, since they contain a small number of loops. The next lemma confirms that we can overcome this by removing the loops and some additional edges while preserving the properties that are necessary for our purposes. The proof is almost entirely present in the literature, hence we put it to an appendix.

\begin{lem}\label{lem:goodgraph}
For every odd prime $p$, there is a simple graph $G(p)$ with the following properties. $G(p)$ is $(p-1)$-regular on $p^2$ vertices, $G(p)$ is $C_4$-free, and each eigenvalue of $G(p)$ besides $p-1$ (which has multiplicity one) is of absolute value at most $\sqrt{4p-5}$.
\end{lem}

The proof of Lemma~\ref{lem:goodgraph} can be found in Appendix~\ref{app:combinatorics}. Observe that when $n=p^2$ is a prime square, Theorem~\ref{thm:kriv} applied to the graph $G(p)$ provided by Lemma~\ref{lem:goodgraph} gives a lower bound on $\overline{H_n(C_4)}$ which together with Lemma~\ref{lem:vertextransitive} finishes the proof. In order to extend this line of thought to general $n$'s, we shall use a theorem ensuring that the primes are dense enough. The strongest such result is due to Baker, Harman and Pintz~\cite{primes}, but for our purposes a more classical version suffices.

\begin{thm}[Ingham~\cite{ingham}]\label{thm:denseenough}
For arbitrary $\eps>0$ and all large enough $n>0$, there is a prime $p$ in the interval $[n,n+n^{5/8+\eps}].$
\end{thm}

\begin{corr}\label{corr:denseenough}
For arbitrary $\eps>0$ and all large enough $n>0$, there is a prime square $p^2$ in the interval $[n,n+n^{13/16+\eps}].$
\end{corr}

\begin{proof}
Let $n>0$ be sufficiently large. By Theorem~\ref{thm:denseenough}, there is a prime $p$ in the interval $[n^{1/2},n^{1/2}+n^{5/16+\eps/2}]$. Then, \[p^2\in[n,n+2n^{13/16+\eps/2}+n^{5/8+\eps}] \subset [n,n+n^{13/16+\eps}]\]
as desired.
\end{proof}

Now we are ready to prove the upper bound of Theorem~\ref{thm:main}. Let $n>0$ be large enough, and let $m$ be the smallest integer such that $n \leq m$ and $m=p^2$ for some prime $p$. Applying Theorem~\ref{thm:kriv} to the $C_4$-free graph $G(p)$ provided by Lemma~\ref{lem:goodgraph}, we see readily
\[\overline{H_m(C_4)}\geq m!\left( \frac{\sqrt{m}-1}{m}\right)^m (1-o(1))^m.\]
As $H_n(C_4)$ is non-decreasing, the previous display combined with Lemma~\ref{lem:vertextransitive} yields
\[H_n(C_4)\leq H_m(C_4)\leq\frac{m!/2}{\overline{H_m(C_4)}}\leq\left(\frac{m}{\sqrt{m}-1}\right)^m(1+o(1))^m
=m^{\frac{m}{2}+o\left(\frac{m}{\log{m}}\right)}.\]
Since by Corollary~\ref{corr:denseenough} we also have $m\leq n+n^{5/6}$, we conclude that
\[H_n(C_4) \leq m^{\frac{m}{2}+o\left(\frac{m}{\log{m}}\right)} \leq n^{\frac{n}{2}+o\left(\frac{n}{\log{n}}\right)}.\]
The proof of Theorem~\ref{thm:main} is complete.

\section{Longer even cycles}\label{sec:longer}

We obtain upper bounds on $H_n(C_{2k})$ for $k>2$ similarly to the $k=2$ case. We shall apply Theorem~\ref{thm:kriv} to suitable $C_{2k}$-free graphs, and use that the necessary number theoretic objects are dense enough. The densest known $C_{2k}$-free graphs were constructed by Lazebnik, Ustimenko and Woldar in \cite{lazebnik}, but they are bipartite and their spectral properties are not yet sufficiently understood. Hence we cannot apply Theorem~\ref{thm:kriv} directly to them. The graphs that have all the necessary properties are the Ramanujan graphs constructed by Margulis~\cite{margulis} and independently by Lubotzky, Phillips and Sarnak~\cite{lubotzky}\footnote{The proof of (4.19)--(4.20) in this celebrated paper is sketchy at one point, so we provide some additional details. The equation before \cite[Lemma~4.4]{lubotzky} can be justified for $n=p^k$ by analyzing the relevant local densities in Siegel's mass formula. This is sufficient for the proof, and one can even obtain a direct proof along these lines. Alternatively, for $n$ coprime with $2q$, the general shape of Fourier coefficients of Eisenstein series reveals that $C(n)=\sum_{d\mid n}dF(d,n)$, where $F:\mathbb{N}\times\mathbb{N}\to\mathbb{C}$ is periodic of period $16q^2$ in both variables. From here, one can finish the proof by a straightforward extension of \cite[Lemma~4.4]{lubotzky}.}.

\begin{thm}[Lubotzky--Phillips--Sarnak~\cite{lubotzky}]
If $p$ and $q$ are unequal primes congruent to $1$ modulo $4$, and $p$ is a quadratic residue modulo $q$, then there is a simple graph $G^{p,q}$ with the following properties. $G^{p,q}$ is $(p+1)$-regular on $q(q^2-1)/2$ vertices, the girth of $G^{p,q}$ is at least $2\log_p{q}$, and each eigenvalue of $G^{p,q}$ besides $p+1$ (which has multiplicity one) is of absolute value at most $2\sqrt{p}$.
\end{thm}

For a suitable pair of primes $(p,q)$ and for $m:=q(q^2-1)/2$, we can apply Theorem~\ref{thm:kriv} to the graph $G^{p,q}$ to get a set of
\[m!\left(\frac{p+1}{m}\right)^m (1+o(1))^m\]
Hamiltonian paths on $m$ vertices. We shall assume that $q>p^k$ in order to guarantee that $G^{p,q}$ is $C_{2k}$-free, while we shall try to work with $p$ as large as possible (for a given $q$) in order to maximize the above number of Hamiltonian paths.

\begin{defi}\label{goodness}
We say that a prime $q\equiv 1\pmod{4}$ is \emph{$(\eps ,k)$-good} if there is a prime $p\equiv 1\pmod{4}$ such that $q\in (p^k,(1+\eps )p^k)$ and $p$ is a quadratic residue modulo $q$. We also call \emph{$(\eps,k)$-good} the corresponding Ramanujan graphs $G^{p,q}$.
\end{defi}

The following lemma states, roughly, that for every pair of positive reals $\eps$ and $k$, the $(\eps,k)$-good graphs are dense enough for our purposes.

\begin{lem}\label{lem:denseenough_pq}
Let $\eps$ and $k$ be arbitrary positive reals. Then for all sufficiently large $n>0$, there is an $(\eps,k)$-good graph $G^{p,q}$ on $m$ vertices with $m\in(n,(1+\eps)n)$.
\end{lem}

The proof of Lemma~\ref{lem:denseenough_pq} relies on the following powerful result of Heath-Brown~\cite{heathbrown}, in which $\left(\frac{n}{m}\right)$ stands for the Jacobi symbol, $\csillag{\textstyle\sum}$ indicates restriction to positive odd square-free integers, and $f \ll_\kappa g$ means that $f=O_\kappa(g)$.

\begin{thm}[Heath-Brown~\cite{heathbrown}]\label{thm:heathbrown}
Let $M$, $N$ be positive integers, and let $a_1,\dotsc,a_N$ be arbitrary complex numbers. Then
for any $\kappa>0$, we have
\[\csillag\sum_{m\leq M}\left|\csillag\sum_{n\leq N}a_n\left(\frac{n}{m}\right)\right|^2\ll_\kappa(MN)^\kappa(M+N)\csillag\sum_{n\leq N}|a_n|^2.\]
\end{thm}

We make use of Theorem~\ref{thm:heathbrown} through the following corollary.

\begin{corr}\label{cor:compl_numb_theor}
Let $\eps$ and $k$ be arbitrary positive reals. Then for all sufficiently large $x>0$, there is an $(\eps,k)$-good prime in the interval $(x,(1+\eps)x)$.
\end{corr}

\begin{proof}
Let $x>0$ be sufficiently large in terms of $\eps$ and $k$, and let us assume that there is no $(\eps,k)$-good prime in the interval $(x,(1+\eps)x)$. We shall derive a contradiction by examining the sum
\begin{equation}\label{eq:considersum}
\sum_{\substack{(1+\eps/2)x<q<(1+\eps)x \\ \text{$q\equiv 1\ (\bmod{4})$ is a prime}}}\
\Biggl|\sum_{\substack{\sqrt[k]{x}<p<\sqrt[k]{(1+\eps/2)x} \\ \text{$p\equiv 1\ (\bmod{4})$ is a prime}}}\left(\frac{p}{q}\right)\Biggr|^2.
\end{equation}
As every prime pair $(p,q)$ occurring in this sum satisfies
\[p,q\equiv 1\pmod{4}\qquad\text{and}\qquad p^k<(1+\eps/2)x<q<(1+\eps)x<(1+\eps)p^k,\]
we have either $\left(\frac{p}{q}\right)=-1$ or $p=q$ by Definition~\ref{goodness} and the initial assumption on $x$.
As a result, by the prime number theorem for arithmetic progressions, the sum \eqref{eq:considersum} can be estimated from below as
\[\gg_{\eps,k}\frac{x}{\log{x}} \left( \frac{x^{1/k}}{\log{x}} \right)^2 =\frac{x^{1+2/k}}{(\log{x})^3}.\]
On the other hand, by Theorem~\ref{thm:heathbrown} applied for
\[\kappa:=\frac{\min(1,k)}{2k+2},\qquad M:=(1+\eps)x,\qquad N:=\sqrt[k]{(1+\eps/2)x},\]
and $a_n$ being the indicator function of the primes $p$ occurring in \eqref{eq:considersum}, the sum \eqref{eq:considersum} can be estimated from above as
\[\ll_{\eps,k}x^{(1+1/k)\kappa}(x+x^{1/k})x^{1/k}\leq x^{1+3/(2k)}+x^{1/2+2/k}.\]
Comparing the above two bounds for \eqref{eq:considersum}, we get a contradiction (for $x>0$ sufficiently large in terms of $\eps$ and $k$).
\end{proof}

\begin{proof}[Proof of Lemma~\ref{lem:denseenough_pq}] It suffices to show that there is an $(\eps,k)$-good prime $q$ such that
\[n<q(q^2-1)/2<(1+\eps)n.\]
A slightly stronger inequality is
\[(2n)^{1/3}+1<q<((1+\eps)2n)^{1/3}.\]
For $n\to\infty$, the ratio of the two sides tends to $(1+\eps)^{1/3}$, which exceeds $1$. Hence, for $n>0$ sufficiently large, Corollary~\ref{cor:compl_numb_theor} implies the existence of a suitable $q$.
\end{proof}

\begin{proof}[Proof of the upper bound of Theorem~\ref{thm:longevencycle}]
Let $\eps\in(0,1/4)$ be fixed, and let $n>0$ be sufficiently large. By Lemma~\ref{lem:denseenough_pq}, there is an $(\eps,k)$-good graph $G^{p,q}$ on $m\in (n,(1+\eps)n)$ vertices. As $G^{p,q}$ is $C_{2k}$-free, it follows from Theorem~\ref{thm:kriv} that
\[\overline{H_m(C_{2k})}\geq m!\left(\frac{p+1}{m}\right)^m(1-o(1))^m.\]
As $H_n(C_{2k})$ is non-decreasing in $n$, the previous display combined with Lemma~\ref{lem:vertextransitive} yields
\[H_n(C_{2k})\leq H_m(C_{2k})\leq\frac{m!/2}{\overline{H_m(C_{2k})}}\leq\left(\frac{m}{p+1}\right)^m(1+o(1))^m.\]
Since $G^{p,q}$ is $(\eps,k)$-good and $m=q(q^2-1)/2$, we can estimate
\[p>\left(\frac{q}{1+\eps}\right)^{1/k}>\left(\frac{(2m)^{1/3}}{1+\eps}\right)^{1/k}>m^\frac{1}{3k}.\]
We conclude, for $n>0$ sufficiently large in terms of $\eps$ and $k$, that
\[H_n(C_{2k})\leq m^{\left(1-\frac{1}{3k}\right)m+o(m)}\leq n^{\left(1-\frac{1}{3k}\right)n+\eps n}.\]
This implies the upper bound of Theorem~\ref{thm:longevencycle}, because $\eps\in(0,1/4)$ is arbitrary.
\end{proof}

\section{$H_n(C_4)$ and the reversing permutations conjecture}\label{sec:connection}

The main driving force behind the study of $H_n(C_{2k})$ is the following conjecture. We say that two permutations $\pi_1$ and $\pi_2$ of $[n]$ are reversing if, as vectors of length $n$, there are two coordinates that contain the same two elements $\{a,b\} \subset [n]$, but in reversed order. Let $RP(n)$ be the maximum number of pairwise reversing permutations of $[n]$.

\begin{conj}[Körner \cite{kornertoldme}]\label{conj:reversing}
There is a constant $C>0$ such that $RP(n)$, the maximum number of pairwise reversing permutations of $[n]$, is at most $C^n$.
\end{conj}

In \cite{evencycle} the following equivalent form of Conjecture~\ref{conj:reversing} was observed. Let $M_{2n}(C_{2k})$ be the maximum number of perfect matchings of $K_{2n}$ where every pairwise union contains a $C_{2k}$.

\begin{conj}[Körner \cite{kornertoldme}]
There is a constant $C>0$ such that $M_{2n}(C_4)$, the maximum number of pairwise $C_4$-creating perfect matchings of $K_{2n}$, is at most $C^n$.
\end{conj}

The best known upper bound for both $RP(n)$ and $M_{2n}(C_4)$ is $n^{\frac{1}{2}n+o(n)}$ due to Cibulka~\cite{cibulka}. The best published lower bound is $RP(n)\geq 8^{\lfloor n/5\rfloor}\geq 1.515^{n-4}$, see \cite{reversefree}. A slightly better lower bound can be obtained by observing that the four ``incomplete permutations''
\[ \begin{array}{c}
(1, 2, 3, *, \dotsc , *) \\
(3, 4, 1, *, \dotsc , *) \\
(2, 1, 4, *, \dotsc , *) \\
(4, 3, 2, *, \dotsc , *) \end{array} \]
are already pairwise reversing. Indeed, this implies $RP(n)\geq 4 RP(n-3)$ for $n\geq 3$, whence $RP(n)\geq 4^{\lfloor n/3\rfloor}\geq 1.587^{n-2}$. In this section we consider questions strongly related to the question of determining $H_n(C_4)$, and one of these will turn out to be strongly related to the reversing permutations conjecture. Two Hamiltonian paths can form a $C_4$ in their union in essentially three ways, so let us introduce a notation for these.

\begin{defi}
Let $H^1,H^2,H^3$ denote the (unique) ways that two Hamiltonian paths can form a $C_4$ in such a way that the longest consecutive chain of edges (in the $C_4$) that is contained in one of the Hamiltonian paths is $1,2,3$, respectively. Let $\mathcal{H}:=\{H^1,H^2,H^3\}$ and for every $S \subset \mathcal{H}$, let $H_n(C_4,S)$ denote the maximum number of pairwise $C_4$-creating Hamiltonian paths, where for each pair there is a $C_4$ that is formed in a way that is in $S$.
\end{defi}

With this notation, $H_n(C_4)= H_n(C_4,\mathcal{H})$, and both Lemma~\ref{claim:newlower} and the original lower bound of Cohen, Fachini and Körner actually establish a lower bound on $H_n(C_4,\{H_2\})$.

We shall use the following lemma, which is implicitly present in the proof of \cite[Theorem~5.2]{oddcycle}, but for the reader's convenience we reprove it here.

\begin{lem}\label{lem:largesystem}
For every $n$, there is a set $J$ of directed Hamiltonian paths of $K_n$ such that no two paths in $J$ are $C_4$-creating in the $H_3$ way and $|J|\geq n!/3^n$.
\end{lem}

\begin{proof}
Let us color the vertices of $K_n$ with colors $1,2,3$ such that exactly $\lceil n/3 \rceil$ vertices are colored $1$ and exactly $\lfloor n/3 \rfloor$ vertices are colored $3$. Let $J$ be the set of directed Hamiltonian paths whose vertices are colored cyclically as $1,2,3,1,2,3,\dotsc$, in the order determined by the path. No two paths in $J$ are $C_4$-creating in the $H_3$ way, since in every path in $J$, any two vertices of distance $3$ have the same color. Hence it suffices to show that $|J|\geq n!/3^n$.

For every directed Hamiltonian path $H$ of $K_n$, there is a unique $3$-coloring as above for which $H \in J$. The size of $J$ does not depend on the particular coloring that we chose, hence it equals $n!$ divided by the number of $3$-colorings. The number of $3$-colorings is at most $3^n$, hence $|J|\geq n!/3^n$ as claimed.
\end{proof}

The following lemma states that when considering a problem $H_n(C_4,S)$ for some $S$, the $H^3$ way can be safely ignored unless we care about exponential factors.

\begin{lem} \label{lem:randomtechnique}
 We have, for every $S \subset \mathcal{H}$,
\[H_n(C_4,S \setminus \{H^3\})\leq H_n(C_4,S)\leq 3^n H_n(C_4,S \setminus \{H^3\}).\]
\end{lem}

\begin{proof}
The lower bound follows readily from the definition of $H_n(C_4,S)$, so we focus on the upper bound. By Lemma~\ref{lem:largesystem}, there is a set $I$ of (undirected) Hamiltonian paths of $K_n$ such that no two paths in $I$ are $C_4$-creating in the $H^3$ way and $|I|\geq 3^{-n}n!/2$. Now let $S \subset \mathcal{H}$, and let $X$ be a set of pairwise $C_4$-creating Hamiltonian paths such that each pair of paths in $X$ is $C_4$-creating in a way that is in $S$. Let $\sigma(I)$ be a version of $I$ where the labels of $K_n$ are permuted by a permutation chosen uniformly from $S_n$. Since
\[\mathbb{E}(|X \cap \sigma(I)|)=|X|\frac{|I|}{n!/2}\geq\frac{|X|}{3^n},\]
there is a relabeled version of $I$ whose intersection with $X$ has size at least $|X|/3^n$. Since no two paths in $I$, or in any relabeled version of $I$, are $C_4$-creating in the $H^3$ way, the proof is complete.
\end{proof}

In the present paper, Hamiltonian paths that are $C_4$-creating only in the $H^3$ way are negligible.

\begin{remark} The quantity $H_n(C_4,\{H^3\})$ is called $P(n,4)$ in \cite{oddcycle}, and the best known upper bound on it is the number of special $3$-colorings considered in the proof of Lemma~\ref{lem:largesystem}, which is of size $3^{n-o(n)}$. It would be interesting to decide whether $H_n(C_4,\{H^3\})$ exceeds $2^n$ for large $n$, since it was proved in \cite{oddcycle} that its natural generalizations, $P(n,k)$ obey an upper bound of size $2^{n+o_k(n)}$.
\end{remark}

Therefore, three possibilities remain for $S$ if we are only interested in the superexponential growth rate: $H_n(C_4,\{H^1\})$, $H_n(C_4,\{H^2\})$, $H_n(C_4,\{H^1,H^2\})$. Two of these are asymptotically answered by the proofs of Theorem~\ref{thm:main} and Lemma~\ref{claim:newlower}:
\[n^{n/2-o(n)} \leq H_n(C_4,\{H^2\}) \leq H_n(C_4,\{H^1,H^2\}) \leq n^{n/2+o(n)}.\]
The only remaining case $H_n(C_4,\{H^1\})$ is strongly related to the reversing permutations conjecture. It was already observed in \cite[Lemma~1]{k=4} that
\[H_n(C_4,\{H^1\}) \geq 2^{n/4-o(n)}RP(n/2).\]
Using the same ideas it is not hard to prove the similar inequality
\[H_{2n}(C_4,\{H^1\}) \geq M_{2n}(C_4)M_n(C_4)M_{\lfloor n/2 \rfloor}(C_4)M_{\lfloor \lfloor n/2 \rfloor/2\rfloor}(C_4) \cdots.\]
This shows that any improvement on the constant $1/2$ in the upper bound provided by Theorem~\ref{thm:main},
\[H_{2n}(C_4,\{H^1\}) \leq (2n)^{\frac{1}{2}(2n)+o(n)}=n^{n+o(n)},\]
would prevent the possibility of a lower bound of the form $n^{n/2-o(n)} \leq M_{2n}(C_4).$

We finish this section by establishing an equivalent form of Conjecture~\ref{conj:reversing}.

\begin{conj}\label{conj:reverseq}
There is a constant $C>0$ such that $ H_n(K_{2,4}) \leq C^n$.
\end{conj}

\begin{lem}\label{lem:equivalence}
Conjecture~\ref{conj:reverseq} is equivalent to Conjecture~\ref{conj:reversing}.
\end{lem}

\begin{proof} We aim to prove that if one of the conjectures fails, then the other one fails, too. Note that if Conjecture~\ref{conj:reverseq} fails, then by the inequality
\[H_m(K_{2,4})H_n(K_{2,4}) \leq H_{m+n}(K_{2,4})\]
and Fekete's Subadditive Lemma, $\sqrt[n]{H_n(K_{2,4})}$ tends to infinity. Similarly, if Conjecture~\ref{conj:reversing} fails, then $\sqrt[n]{RP(n)}$ tends to infinity.

Assume that Conjecture~\ref{conj:reversing} is false. For each $n\geq 1$, there exists a set $S$ of pairwise reversing permutations of $[n]$ such that $\sqrt[n]{|S|}\to\infty$. Starting from $S$, we construct a large set $T$ of pairwise $K_{2,4}$-creating Hamiltonian paths of $K_{4n}$ as follows. We arrange the elements of $[4n]$ as in Figure~\ref{fig:paths}. We fix the position of the unlabeled vertices and vary the position of the labeled vertices according to the elements of $S$. The resulting set $T$ of Hamiltonian paths will be pairwise $K_{2,4}$-creating, because $S$ is pairwise reversing. Moreover, $\sqrt[4n]{|T|}\to\infty$, because $\sqrt[n]{|S|}\to\infty$. Hence Conjecture~\ref{conj:reverseq} is false.

Assume that Conjecture~\ref{conj:reverseq} is false. For each $n\geq 3$, there exists a set $X$ of pairwise $K_{2,4}$-creating Hamiltonian paths of $K_n$ such that $\sqrt[n]{|X|}\to\infty$. A pair of Hamiltonian paths in $X$ can create a $K_{2,4}$ in several ways, hence we shall first ``filter'' $X$ so that each ``surviving'' pair creates a $K_{2,4}$ in a fixed way. Observe that if a $K_{2,4}$ is created by two Hamiltonian paths in such a way that none of the paths contains three consecutive edges in the $K_{2,4}$, then both paths contribute two vertex disjoint stars on $3$ vertices, see Figure~\ref{fig:onlyway}.  We shall call a $K_{2,4}$ in the union of two Hamiltonian paths \textit{good} if it is created as in Figure~\ref{fig:onlyway}. If two Hamiltonian paths create a $K_{2,4}$ that is not good, one of the paths contributes at least three consecutive edges to the $K_{2,4}$, and the other path necessarily completes these three edges into a $C_4$. Hence the two paths are $C_4$-creating in the $H_3$ way. Now we proceed similarly as in the proof of Lemma~\ref{lem:randomtechnique}. By Lemma~\ref{lem:largesystem}, there is a set $I$ of (undirected) Hamiltonian paths of $K_n$ such that no two paths in $I$ are $C_4$-creating in the $H^3$ way and $|I|\geq 3^{-n}n!/2$. Then, each pair of paths in $I$ either creates a good $K_{2,4}$ or it does not create a $K_{2,4}$ at all. Let $\sigma(I)$ be a version of $I$ where the labels of $K_n$ are permuted by a permutation chosen uniformly from $S_n$. Since
\[ \mathbb{E}(|X \cap \sigma(I)|)=|X|\frac{|I|}{n!/2}\geq\frac{|X|}{3^n},\]
there is a ``filtered'' set $T:=X\cap\sigma(I)$ of Hamiltonian paths of $K_n$ such that $\sqrt[n]{|T|}\to\infty$ and each pair of paths in $T$ creates a good $K_{2,4}$. Now we build a set of pairwise reversing vectors from $T$. For each path $H \in T$, let us partially label the edges of $K_n$ as follows. If an edge $e\in E(K_n)$ creates a triangle with $H$, then $e$ is labeled by the vertex of the triangle which is not incident to $e$, while the other edges are not labeled. By fixing an ordering of $E(K_n)$, we can convert the labelings into partially filled vectors satisfying the following properties. Each vector has length $\binom{n}{2}$, but only $n-2$ coordinates are filled. The filled coordinates contain $n-2$ different elements, and by the choice of $T$, the vectors are pairwise reversing. Let us denote this set of vectors by $V(n)$, then clearly $\sqrt[n]{|V(n)|}\to\infty$.

Now let us fix $n\geq 3$ for a moment, and let $m\geq\binom{n}{2}$ be arbitrary. Starting from any $v\in V(n)$, we can construct $RP(m-(n-2))$ pairwise reversing permutations $\pi$ of $[m]$ by first adding $m-\binom{n}{2}$ unfilled coordinates to $v$, and then filling the $m-(n-2)$ unfilled coordinates of the augmented vector appropriately. These sets of permutations (coming from the various $v\in V(n)$) are pairwise disjoint, and their union is still pairwise reversing by the reversing property of $V(n)$. This proves the inequality
\[RP(m)\geq |V(n)|\cdot RP(m-(n-2)),\qquad m\geq\tbinom{n}{2}.\]
From here it is straightforward to deduce that
\[\limsup\sqrt[m]{RP(m)}\geq\sqrt[n-2]{|V(n)|}.\]
On the right hand side, $\sqrt[n-2]{|V(n)|}$ is not bounded in $n$, so $\sqrt[m]{RP(m)}$ is not bounded in $m$. Hence Conjecture~\ref{conj:reversing} is false.
\end{proof}

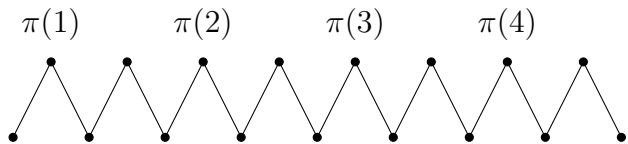
\begin{figure}
\begin{tikzpicture}[scale=1]
\tikzstyle{vertex}=[draw,circle,fill=black,minimum size=3,inner sep=0]
\foreach \j in {1,...,8}
{
\node[vertex] (y\j) at (-\j,0){};
\node[vertex] (x\j) at (-\j-0.5,1){};
\draw (-\j,0)--(-\j-0.5,1)--(-\j-1,0);
}
\node[vertex] (y9) at (-9,0){};
\node at ( -2.5,1.5)   () {$\pi(4)$};
\node at ( -4.5,1.5)   () {$\pi(3)$};
\node at ( -6.5,1.5)   () {$\pi(2)$};
\node at ( -8.5,1.5)   () {$\pi(1)$};
\end{tikzpicture}
\caption{$K_{2,4}$-creating Hamiltonian paths from pairwise reversing permutations.}
\label{fig:paths}
\end{figure}

\begin{figure}
\begin{tikzpicture}[scale=1]
\tikzstyle{vertex}=[draw,circle,fill=black,minimum size=3,inner sep=0]
\node[vertex] (x1) at (0,0){};
\node[vertex] (x2) at (1,0){};
\node[vertex] (x3) at (2,0){};
\node[vertex] (x4) at (3,0){};
\node[vertex] (y1) at (0.5,1){};
\node[vertex] (y2) at (2.5,1){};
\draw (x1)--(y1)--(x2);
\draw (x3)--(y2)--(x4);
\draw[dashed] (x1)--(y2)--(x2);
\draw[dashed] (x3)--(y1)--(x4);
\end{tikzpicture}
\caption{The only way how two Hamiltonian paths can create a \emph{good} $K_{2,4}$. The dashed edges belong to one of the paths and the rest to the other.}
\label{fig:onlyway}
\end{figure}
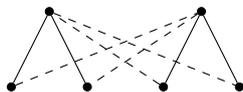

\section{Concluding remarks}

Roughly speaking, Theorems~\ref{thm:oldk=4} and \ref{thm:oldevenk} ``draw their power'' from the fact that there are $C_4$-free (or $C_{2k}$-free) graphs with many perfect matchings. Their improvements in the present paper, Theorems~\ref{thm:main} and \ref{thm:longevencycle}, ``draw their power'' from the fact that there are $C_4$-free (or $C_{2k}$-free) graphs with many Hamiltonian paths.

In the proof of Theorem~\ref{thm:main}, we employed Lemma~\ref{lem:vertextransitive} to establish the inequality
\begin{equation}\label{eq:thisisthefractionalLOL}
H_n(C_4) \leq \frac{n!/2}{\overline{H_n(C_4)}},
\end{equation}
and then we used Hamiltonian paths of a $C_4$-free graph to get an asymptotically optimal lower bound on $\overline{H_n(C_4)}$. In the following subsection we argue that using $G$-free graphs to prove lower bounds on $\overline{H_n(G)}$ does not always yield asymptotically optimal results.

\subsection{Constructions using $G$-free graphs}
From Theorem~\ref{thm:main} and \cite{oddcycle} it follows that for $k\in\{3,4,5\}$, an asymptotically best construction for $\overline{H_n(C_k)}$ can be obtained by choosing a $C_k$-free graph with enough Hamiltonian paths. This is not the case in general for $\overline{H_n(G)}$, a counterexample being $G=K_{3,3}$. Indeed, Brown~\cite{brown} and Füredi~\cite{furedi} proved that
\begin{equation}\label{eq:BrownFuredi}
ex(n,K_{3,3})=\frac{1}{2}n^{5/3}+o(n^{5/3}),
\end{equation}
which implies that the number of Hamiltonian paths in a $K_{3,3}$-free graph on $n$ vertices is at most $n^{\frac{2}{3}n+o(n)}$. To see this, observe that the number of Hamiltonian paths is at most $n$ times the product of the degrees, and then apply the inequality of arithmetic and geometric means coupled with \eqref{eq:BrownFuredi}. However, an asymptotically larger construction exists. Let $\overrightarrow{G}_3$ be a complete tripartite, directed graph with parts $X_1,X_2,X_3$ of size as equal as possible and the edges between $X_i$ and $X_j$ are directed towards $X_j$ if $j-i\equiv 1 \pmod{3}$. The number of directed Hamiltonian paths in $\overrightarrow{G}_3$ (in which there are no vertices of indegree or outdegree $2$) is $n^{n+o(n)}$, and it is an easy exercise that no two such paths contain a $K_{3,3}$ in their union.

\subsection{Using Lemma~\ref{lem:vertextransitive}}

By the proof of Lemma~\ref{lem:vertextransitive}, we see that the right hand side of \eqref{eq:thisisthefractionalLOL} is actually equal to the fractional relaxation of $H_n(C_4)$. Hence Lemma~\ref{lem:vertextransitive} is useful if there is no large gap between the clique and the fractional clique number of the underlying graph. The gap between the clique number and the fractional clique number of a vertex-transitive graph can be arbitrarily large, a longer discussion about the size of the gap compared to the number of vertices can be found in \cite{MOquestion}.

For the problem of determining $RP(n)$, Conjecture~\ref{conj:reversing} would imply such a large gap: Cibulka~\cite{cibulka} proved that $n^{\frac{n}{2}-o(n)}\leq\overline{M_{2n}(C_4)}\leq n^{\frac{n}{2}+o(n)}$, hence no improvement on the upper bound $M_{2n}(C_4) \leq n^{\frac{n}{2}+o(n)}$ can be made using only Lemma~\ref{lem:vertextransitive} and bounds on $\overline{M_{2n}(C_4)}$.

\appendix

\section{Proof of Lemma~\ref{lem:goodgraph}}\label{app:combinatorics}

Let $p$ be an odd prime, and let $\mathbb{F}_p$ denote the finite field of $p$ elements. Let the vertex set of the graph $\tilde G=\tilde G(p)$ be the vector space $\F^2$, and let two vertices $(a,c),(b,d) \in \F^2$ be adjacent in $\tilde G$ if and only if $ab=c+d$. This is a well-known construction for a dense $C_4$-free graph, and the spectral properties of $\tilde G$ have already been studied, see for example the work of Mubayi--Williford~\cite{mubayi} and Solymosi~\cite{solymosi}. Thus it is already established in the literature that $\tilde G$ satisfies all our requirements except that it contains loops. Getting rid of the loops in a way that the resulting graph is still regular with a large spectral gap is easy, and here we present one way of doing it.

We partition the vertex set of $\tilde G$ into the affine subspaces $P_a$ ($a\in\F$), where $P_a$ consists of all pairs from $\F^2$ whose first coordinate is $a$. We delete from $\tilde G$ all the edges between the vertices of each $P_a$ (in particular, we delete all loops), and we denote by $G=G(p)$ the resulting graph. We claim that this graph satisfies all the conditions of Lemma~\ref{lem:goodgraph}, and for the proof we collect first some basic properties of $\tilde G$.

\begin{lem}\label{claim:calculations} The following statements hold for the graph $\tilde G$.
\begin{enumerate}
\item For distinct $a,b\in\F$, the edges of $\tilde G$ between $P_a$ and $P_b$ form a perfect matching.
\item For $a\in\F$, the subgraph of $\tilde G$ spanned by $P_a$ consists of a perfect matching on $p-1$ vertices and the remaining single vertex with a loop.
\item For distinct $a,b\in\F$, every pair consisting of a vertex from $P_a$ and a vertex from $P_b$ has a unique common neighbor in $\tilde G$.
\item For $a\in\F$, no pair of different vertices from $P_a$ has a common neighbor in $\tilde G$.
\end{enumerate}
\end{lem}

\begin{proof} We prove the statements one by one.
\begin{enumerate}
\item Two vertices $(a,c)\in P_a$ and $(b,d)\in P_b$ are adjacent in $\tilde G$ if and only if $ab=c+d$. For given distinct $a,b\in\F$, this equation determines a pairing $c\leftrightarrow d$ on $\F$.
\item Two vertices $(a,c)\in P_a$ and $(a,d)\in P_a$ are adjacent in $\tilde G$ if and only if $a^2=c+d$. For given $a\in\F$, this equation determines a pairing $c\leftrightarrow d$ on $\F$, and $a^2/2$ is the unique element in $\F$ whose pair is itself.
\item Two vertices $(a,c)\in P_a$ and $(b,d)\in P_b$ have $(x,y)\in\F^2$ as a common neighbor in $\tilde G$ if and only if $ax=c+y$ and $bx=d+y$. As $a\neq b$, the unique solution of these equations is $x=\frac{c-d}{a-b}$ and $y=\frac{bc-ad}{a-b}$.
\item Two vertices $(a,c)\in P_a$ and $(a,d)\in P_a$ have $(x,y)\in\F^2$ as a common neighbor in $\tilde G$ if and only if $ax=c+y=d+y$. These equations have no solution when $c\neq d$.
\end{enumerate}
\end{proof}

We shall deduce Lemma~\ref{lem:goodgraph} from Lemma~\ref{claim:calculations}. The graph $G$ is $(p-1)$-regular by parts~(1) and (2) of Lemma~\ref{claim:calculations}, and it is $C_4$-free by parts~(3) and (4) of Lemma~\ref{claim:calculations}. It remains to prove that the eigenvalue $p-1$ of $G$ has multiplicity one, and every other eigenvalue has absolute value at most $\sqrt{4p-5}$. In order to analyze the spectrum of $G$, we identify $\F$ with $\{0,1,\dotsc,p-1\}$, and we order it accordingly: $0<1<\dotsb<p-1$. Then, we order $\F^2$ lexicographically: $(a,c)<(b,d)$ if and only if $a<b$ or $a=b$ and $c<d$.

We consider the adjacency matrix $A(\tilde G)$ of $\tilde G$, where the $i$-th row (and similarly the $i$-th column) corresponds to the $i$-th vertex of $\tilde G$ in the lexicographic order. This is a $p^2\times p^2$ symmetric matrix whose square has a simple structure in terms of $p \times p$ blocks. Indeed, the $(i,j)$-entry of $A(\tilde G)^2$ is the number of common neighbors in $\tilde G$ of the (not necessarily distinct) $i$-th and $j$-th vertices, hence by Lemma~\ref{claim:calculations} we get that
\begin{equation}\label{eq:matrix}
A(\tilde G)^2=\begin{pmatrix}
p I_{p \times p} & \mathbb{1}_{p \times p} & \cdots & \mathbb{1}_{p \times p} \\
\mathbb{1}_{p \times p} &p I_{p \times p} & \cdots & \mathbb{1}_{p \times p} \\
\vdots & \vdots & \ddots & \vdots \\
\mathbb{1}_{p \times p} & \mathbb{1}_{p \times p} & \cdots & p I_{p \times p}
\end{pmatrix},
\end{equation}
where $I_{p \times p}$ is the identity matrix, and $\mathbb{1}_{p \times p}$ is the matrix containing only ones.

The structure of $A(G)^2$ can be understood by considering the difference between common neighbors in $\tilde G$ and $G$. Let $\mathbb{1}_{p \times p}^{\ominus 2}$ be the class of matrices with exactly $p-2$ ones in every row and column, and zeros elsewhere. By changing $\tilde G$ to $G$, the diagonal blocks in equation \eqref{eq:matrix} change from $p I_{p \times p}$ into $(p-1) I_{p \times p}$, while the off-diagonal blocks change from $\mathbb{1}_{p \times p}$ into matrices lying in $\mathbb{1}_{p \times p}^{\ominus 2}$. Indeed, if $a,b\in\F$ are distinct, then for every vertex $(a,c)\in P_a$ there are precisely two vertices $(b,d)\in P_b$ which have no common neighbor with $(a,c)$ in $G$. These are, within $\tilde G$, the neighbor in $P_b$ of the neighbor in $P_a$ of $(a,c)$, and the neighbor in $P_b$ of the neighbor in $P_b$ of $(a,c)$. Note that these two vertices $(b,d)\in P_b$ are distinct by part~(3) of Lemma~\ref{claim:calculations}. Explicitly, these are the two vertices $(b,d)\in P_b$ such that $\frac{c-d}{a-b}$ equals $a$ or $b$ (cf.\ proof of part~(3) of Lemma~\ref{claim:calculations}). To summarize, the analogue of equation \eqref{eq:matrix} is the relation
\begin{equation}\label{eq:matrix2}
A(G)^2\in\begin{pmatrix}
(p-1) I_{p \times p} & \mathbb{1}_{p \times p}^{\ominus 2} & \cdots & \mathbb{1}_{p \times p}^{\ominus 2} \\
\mathbb{1}_{p \times p}^{\ominus 2} &(p-1) I_{p \times p} & \cdots & \mathbb{1}_{p \times p}^{\ominus 2} \\
\vdots & \vdots & \ddots & \vdots \\
\mathbb{1}_{p \times p}^{\ominus 2} & \mathbb{1}_{p \times p}^{\ominus 2} & \cdots & (p-1) I_{p \times p}
\end{pmatrix}.
\end{equation}

Let $V$ be the space of real column vectors of length $p^2$, and let $v\in V$ be the column vector with entries $1$. Note that $v$ is an eigenvector of $A(G)$ with eigenvalue $p-1$, because $G$ is $(p-1)$-regular. By the spectral theorem for symmetric matrices, it suffices to show that every eigenvector $w\in V$ of $A(G)$ orthogonal to $v$ has eigenvalue $\lambda\leq\sqrt{4p-5}$. Clearly, these conditions imply that $w$ is an eigenvector of the matrix
\[S(G):=A(G)^2-\mathbb{1}_{p^2 \times p^2}\]
with eigenvalue $\lambda^2$. By \eqref{eq:matrix2} we also see that, in each row of $S(G)$, the sum of the absolute values of the entries equals
$(p-2)+(p-1)+2(p-1)=4p-5$. Therefore, by a well-known inequality on the spectral radius (see \cite[Proposition~7.6]{serre}), we conclude that
\[\lambda^2\leq{\|S(G)\|}_\infty=4p-5.\]
The proof of Lemma~\ref{lem:goodgraph} is complete.

\end{document}